%% file: AlgSections.tex
\RequirePackage[final]{graphicx}
\documentclass[oneside, dvipsnames, svgnames, table]{amsart}

\usepackage[T1]{fontenc}
\usepackage[utf8]{inputenc}

\usepackage{fourier}
\usepackage[foot]{amsaddr}
\usepackage{adjustbox}
\input{tbjwHeader}
\usepackage{comment}

\definecolor{highlight}{rgb}{0.99,0.96,0.94}

\usepackage[backref=true,firstinits=true,backend=biber,style=numeric,url=false]{biblatex}
\newcommand{\sebg}[2][]{\ifdraft{\todo[linecolor=Blue,backgroundcolor=Blue!25,bordercolor=Blue,#1]{#2---Seb G.}}{}}

\newcommand{\benw}[2][]{\ifdraft{\todo[linecolor=Green,backgroundcolor=Green!25,bordercolor=Green,#1]{#2---Ben W.}}{}}

\newcommand{\GW}{\mathbf{GW}}

\newcommand{\kq}{\mathbf{kq}}

\newcommand{\one}{\mathbb{1}}
\newcommand{\taut}{\mathrm{taut}}

\author{W.~S.~Gant and Ben Williams}
\address{Department of Mathematics, the University of British Columbia\\
  1984 Mathematics Rd \\
  Vancouver BC V6T 1Z2\\
Canada.}
\email[W.~S.~Gant]{wsgant@math.ubc.ca}
\email[B.~Williams]{tbjw@math.ubc.ca}

\begin{document}
\title{Free summands of stably free modules}

\begin{abstract}
 Let $R$ be a commutative ring. One may ask when a general $R$-module $P$ that satisfies $P \oplus R \iso R^n$ has a free
 summand of a given rank. M.~Raynaud 
  translated this question into one about sections of certain maps between Stiefel varieties: if $V_r(\A^n)$ denotes the
   variety $\GL(n) / \GL(n-r)$ over a field $k$, then the projection $V_r(\A^n) \to V_1(\A^n)$ has a section if
  and only if the following holds: any module $P$ over any $k$-algebra $R$ with the property that $P \oplus R \iso R^n$
  has a free summand of rank $r-1$. Using techniques from $\A^1$-homotopy theory, we characterize those $n$ for which
  the map $V_r(\A^n) \to V_1(\A^n)$ has a section in the cases $r=3,4$ under some assumptions on the base field.

  We conclude that if $P \oplus R \iso R^{24m}$ and $R$ contains the field of rational numbers, then $P$ contains a free
  summand of rank $2$. If $R$ contains a quadratically closed field of characteristic $0$, or the field of real numbers,
  then $P$ contains a free summand of rank $3$. The analogous results hold for schemes and vector bundles over them.
\end{abstract}
\thanks{We acknowledge the support of the Natural Sciences and Engineering Research Council of Canada (NSERC), RGPIN-2021-02603.}

\maketitle

\section{Introduction}
\label{sec:introduction}

Suppose $R$ is a commutative ring, and $n,r$ are integers satisfying $0 \le r \le n$. An $R$-module $P$ is \emph{stably free of type $(n,r)$} if there exists an isomorphism of $R$-modules:
\begin{equation} \label{eq:sfdef}
P \oplus R^{n-r} \iso R^n. 
\end{equation}
The most important nontrivial instance is that of $r=n-1$, since any isomorphism \eqref{eq:sfdef} entails an isomorphism
\[ 
 (P \oplus R^{n-r-1}) \oplus R \iso R^n,
\]
so that $P \oplus R^{n-r-1}$ is stably free of type $(n,n-1)$. In a sense that is made precise in \cite[Th\'eor\`eme 6.5]{Raynaud1968}, a general stably free module of type $(n,n-1)$ does not admit a free summand of large rank, and is \textit{a fortiori} not free.

We set up some notation. Let $k$ be a commutative base ring. The letter $k$ may be omitted from the notation where no confusion can arise. Fix a pair of integers $(n,r)$ where $0 \le r \le n$.
As set out in \cite{Raynaud1968}, there is a commutative $k$-algebra $A_{n,n-r}$ and a stably free module $P_{n,n-r}$ of type $(n,r)$ over $A_{n,n-r}$ that is universal: for any commutative $k$-algebra $R$ and any stably free $R$-module $P$ of type $(n,r)$,  there is a ring homomorphism $A_{n,n-r} \to R$ such that $P \cong R \otimes_{A_{n,n-r}} P_{n,n-r}$ as $R$-modules. 

There exist a sequence of positive integers called \emph{James numbers} and written $b_2, b_3, \dots$, which were defined by James \cite{James1958}, and calculated by \cite{Atiyah1960} and \cite{Adams1965}.
Explicitly, they are described by their $p$-adic valuations, $v_p$, for all primes $p$:
\[ v_p(b_q) = \begin{cases} \max\left\{s + v_p(s) \:\Big|\: 1 \le s \le \lfloor \frac{ q-1 }{p-1} \rfloor \right\}, \quad \text{ if  $q \ge p$;}  \\ 
0 \qquad \text{ otherwise.} \end{cases} \]
The first few James numbers may easily be listed:
\[ b_2 = 2; \quad b_3=b_4=2^3 3=24; \quad b_5=2^6 3^2 5=2880. \]

The following is the module-theoretic content of \cite[Th\'eor\`eme 6.5]{Raynaud1968}.
\begin{theorem*}[Raynaud]
   If $k$ is a field of characteristic $0$, then the universal stably free module $P_{n,n-1}$ of type $(n,n-1)$ does not admit a free summand of rank $q-1$, except possibly if $b_q \mid n$.
\end{theorem*}
This result does not make any assertion about the situation when $b_q \mid n$. It is well known that if $n$ is even (\textit{viz.} divisible by $b_2$), then $P_{n,n-1}$ admits a free summand of rank $1$, as is shown in Example \ref{ex:r=2Sec} below. The cases of larger $q$ are more obscure.

In this paper, we prove the following. This is the module-theoretic content of Theorem \ref{th:schThryMain}.
\begin{theorem*} \label{th:main}     
 Suppose $R$ is a commutative ring containing $\Q$. Let $n$ be a natural number. If $P$ is a stably free module of type $(24n, 24n-1)$, then $P$ admits a free summand of rank $2$. 
 If $R$ contains a subfield of $\R$ that has a unique quadratic extension (up to isomorphism) then $P$ admits a free summand of rank $3$.
\end{theorem*}

The result also applies with a $k$-scheme $X$ playing the part of the $k$-algebra $R$. In this case the stably free module becomes a stably trivial vector bundle.

\begin{remark} \label{rem:dontNeedQuadClosed}
    The condition on $R$ in the second part of the theorem above is satisfied if $R$ contains a characteristic-$0$ quadratically closed field $F$. To see why, we argue as follows. The field $F$ contains a quadratically closed subfield $E$ consisting of algebraic numbers, and $E$ may be embedded in $\C$. Let $i$ denote a square root of $-1\in\C$, and let $z \mapsto \bar z$ denote ordinary complex conjugation. The subfield of $E$ fixed by conjugation is $E'=E\cap \R$. We claim that $E'$ meets the conditions of the theorem. 
    
    By construction $E' \subseteq \R$. When viewed as a vector space over $E'$, the field $E$ decomposes as a direct sum of eigenspaces for complex conjugation, so that we see $E=E' \oplus i E'$, which implies that $E=E'(i)$. Since $E$ is quadratically closed, we deduce that it is the unique quadratic extension of $E'$, up to isomorphism. 
\end{remark}


\subsection*{Geometry}

The methods of \cite{Raynaud1968} are geometric and homotopy-theoretic, and so too are the methods of this paper. We specialize to the case where the base ring $k$ is a field of characteristic $0$.

If $a \le b$ are natural numbers, we embed the group scheme $\GL(a)$ into $\GL(b)$ by 
\[ A \mapsto \begin{bmatrix} I_{b-a} & 0 \\ 0 & A
\end{bmatrix}.\]
We let
\[ V_r(\A^n) = \GL(n)/\GL(n-r)\] denote the Stiefel variety. There is a canonical projection $\rho: V_r(\A^n) \to V_1(\A^n)$. 

The ring $A_{n,n-r}$ that we referred to previously is the coordinate ring of $V_r(\A^n)$, and $P_{n,n-r}$ is a module over it. As a special case of \cite[Proposition 2.4]{Raynaud1968},
the map $\rho: V_r(\A^n) \to V_1(\A^n)$ has a section if and only if $P_{n,n-1}$ has a free summand of rank $r-1$.  Therefore, the question of whether all stably free modules of type $(n,n-1)$ (over $k$-algebras) admit free summands of rank $r-1$ is equivalent to:
\begin{question}\label{qu:1}
    Does the morphism of $k$-schemes $\rho: V_r(\A^n) \to V_1(\A^n)$ admit a section?
\end{question}

James answered the topological analogue of this question in \cite{James1958}: if $W_r(\C^n)$ denotes the complex Stiefel manifold of orthonormal $r$-frames in $\C^n$, then projection onto the first frame $\rho_\C: W_r(\C^n) \to S^{2n-1}$ has a continuous section if and only if $n$ is divisible by the integer $b_r$.

Using Steenrod operations in \'{e}tale cohomology, Raynaud showed that, over a characteristic $0$ field $k$, the map $\rho: V_r(\A^n) \to V_1(\A^n)$ does not have a section if $n$ is not divisible by $b_r$ (this is the geometric content of \cite[Th\'{e}or\`{e}me 6.5]{Raynaud1968} above). 

\subsection*{Method}

The method of proof in this paper is to convert the algebro-geometric problem of Question \ref{qu:1} to a problem in $\A^1$-homotopy theory. This mimics how an analogous question about vector bundles on topological spaces has been fully solved by the methods of homotopy theory and the calculation of certain periodicities, by \cite{James1958}, \cite{Toda1955}, \cite{Atiyah1960} and \cite{Adams1965}. The structure of the topological argument is to reduce the problem to determining whether a certain class in $\pi_{2n-2}(W_{r-1}(\C^{n-1}))$ vanishes, which then may be calculated using techniques developed by Adams.

The analogous obstruction in $\A^1$-homotopy theory is defined in Notation \ref{not:defBeta}. In Proposition \ref{pr:equivOfSections}, we show that the existence of a section is equivalent to the vanishing of the obstruction. One direction of this is trivial. To construct the section knowing that the obstruction vanishes, however, we  use the Lindel--Popescu Theorem  (\cite{Lindel1981}) about homotopy invariance of algebraic vector bundles, an observation of \cite[Proposition 2.4]{Raynaud1968}, and the result of \cite[Theorem 2.4.2]{Asok2018} which relates abstract morphisms in the $\Aone$-homotopy category to naive homotopy classes of morphisms of schemes.

Having converted the problem to one in $\Aone$-homotopy theory, we now solve it in Propositions \ref{pr:r=3New} and \ref{pr:r=4} by using realization methods: comparing the global sections of $\Aone$-homotopy sheaves of spaces with the homotopy groups of their real- or complex-realizations. For the Stiefel varieties at issue, we can prove that the comparison maps in question are isomorphisms. In this way, we show that the answers to the algebraic question is ``the same'' as the answer for complex vector bundles.

The major inputs are the calculations of the stable homotopy sheaves of spheres by \cite{Rondigs2019} and \cite{Rondigs2024} and the Freudenthal suspension theorem of \cite{Asok2024}, by which we can understand unstable $\Aone$-homotopy sheaves of spheres, from which we can calculate the unstable $\Aone$-homotopy sheaves of Stiefel varieties.

Finally, to establish the stronger form of our main theorem when $k$ is quadratically closed, or is a subfield of $\R$ admitting only one quadratic extension, we use an argument from \cite{James1958}, now applied to the $\Aone$-homotopy sheaves.

\subsection*{Positive and mixed characteristic}

The morphism $\rho : V_r(\A^n) \to V_1(\A^n)$ may be defined over $\Z$, and over $\Z$ the spaces $V_r(\A^n)$ still represent stably free modules. In Proposition \ref{pr:r=3New}, we prove that when $r=3$ and $n=24m$, the base-change morphism $\rho$ over $\Q$ has a section. In particular, this means that $\rho$ has a section over $\Z[\frac{1}{N}]$, where $N$ is some positive integer: only finitely many primes need to be inverted in order to construct the section. We remark that \textit{a priori} the integer $N$ may depend on $n$.

In particular, we can declare that for any given $n=24m$, there exists some integer $N$ so that if $R$ is a ring in which $N$ is invertible and $P$ is a stably free $R$-module of type $(n,n-1)$, then $P$ has a free summand of rank $2$. We do not know the smallest possible positive integer $N$, although $N=6$ is a plausible conjecture.

\medskip

One might also wonder whether the methods used to prove Proposition \ref{pr:r=3New} can be made to work over a prime field $\FF_p$ of characteristic $p$. There are two difficulties: most seriously, the main results \cite{Rondigs2019} and \cite{Rondigs2024} do not fully determine the homotopy groups in question: the $p$-torsion is not determined, since the exponential characteristic of the ground field must be inverted throughout.

Secondly, the complex realization functor must be replaced by $\ell$-\'etale realization (see e.g., \cite{Isaksen2004a}), which takes values in $\ell$-complete spaces or pro-spaces where $\ell \neq p$. We expect our obstructions to lie in groups isomorphic to $\Z/(24)$ contingent on a strengthening of \cite{Rondigs2019} and \cite{Rondigs2024} that holds without inverting the exponential characteristic. Therefore we will have to use realization for the primes $2, 3$, which implies that the constraint $p \neq 2,3$ is probably unavoidable.

\section{Homotopy-theory conventions and notation}
\label{sec:conventions-notation}

For the rest of the paper, we work over a field $k$ of characteristic $0$. Other
restrictions may be imposed on $k$ from time to time. Unless otherwise stated, all schemes appearing are
$k$-schemes. The category of \emph{motivic spaces} over $k$ is $\sPre(\Sm_k)$, the category of simplicial presheaves of
sets on $\Sm_k$, which is itself the category of smooth finite type separated $k$-schemes. We use the homotopy theory of
\cite{Morel1999}. The notation $\cat{H}(k)$ is used for the homotopy category. A \emph{pointed object} consists of an object $X$ and a morphism $\Spec k \to X$. There is a homotopy theory of pointed objects, and the associated homotopy category is denoted $\cat{H}(k)_\bullet$. The notation $X_+$ is used to denote the addition of a disjoint basepoint to $X$.

\subsection{Homotopy sheaves}
\label{sec:homotopy-sheaves}

The paper makes extensive use of the $\Aone$-homotopy theory of spheres. If $p,q$ are nonnegative integers, then we define 
\[ S^{p+q\alpha} = S^{p+q, q} = S^p \wedge (\A_k^1 \sm \{0\})^{\wedge q} \]
where $S^p$ is the ordinary simplicial sphere.

If $X$ is a pointed object of $\sPre(\Sm_k)$, then we write
\[ \bpi_{p+q\alpha}(X) = \bpi_{p+q, q}(X) \]
for the Nisnevich sheaf associated to the functor
\[ \Sm_k^\op \to \cat{Set}: \: U \mapsto [U_+ \wedge S^{p+q\alpha}, X]. \]
When $p \ge 1$, then $\bpi_{p+q\alpha}(X)$ is a sheaf of groups, and of abelian groups if $p \ge 2$.

The symbols $\K^\MW_n$ and $\K^\M_n$ will be used for the unramified sheaves constructed in \cite[Section 3.2]{Morel2012}. We make use of the \emph{contraction} $\mathbf{A} \mapsto \mathbf{A}_{-1}$, for which we refer to \cite[p.~33]{Morel2012} and \cite[Theorem 6.13]{Morel2012}: which implies that
\[ \bpi_{p+q\alpha}(X)_{-1} \iso \bpi_{p+(q+1)\alpha}(X). \]

\subsection{Naive homotopy}
\label{sec:naive-homotopy-1}

If $X$ is a scheme, we write $j_0, j_1: X \to X \times \A^1_k$ for the closed inclusions at $0$, $1$ respectively.

Two morphisms $f_0, f_1: X \to Y$ are said to be \emph{naively homotopic} if there is a morphism
$H : X \times \Aone_k \to Y$ such that $H \circ j_0 = f_0$ and $H \circ j_1 = f_1$.

\subsection{Pointed homotopy}
\label{sec:pointed-homotopy}

If $X$, $Y$ are objects of $\sPre(\Sm_k)$, then we will write $\{X,Y\}$ for the set of morphisms $X \to Y$ in the homotopy category $\cat{H}(k)$. If $X$ and $Y$ are pointed objects of $\sPre(\Sm_k)$, then the notation $[X,Y]$ will be used to denote the set of morphisms $X\to Y$ in $\cat{H}(k)_\bullet$. There is a natural bijection
\[ \{X, Y\} \leftrightarrow [X_+,Y]. \]

\begin{proposition} \label{pr:basepointsMayBeIgnored}
  Suppose $X,Y$ are two pointed objects of $\sPre(\Sm_k)$, and $Y$ is $\Aone$-simply-connected. The natural map
  \[[X,Y] \to \{X,Y\} \]
  is a bijection.
\end{proposition}
\begin{proof}
  We work in the $\Aone$-injective model category. There is nothing to be lost in assuming $X$ is cofibrant and $Y$ is fibrant.

    There is a cofibre sequence of pointed objects
    \begin{equation}\label{eq:cofSeq}
         S^0 \to X_+ \to X.
    \end{equation}

We write $\uHom$ for the mapping object of $\sPre(\Sm_k)$ and $\uHom_\bullet$ for its pointed analogue. Apply $\uHom_\bullet(-, Y)$ to \eqref{eq:cofSeq} to obtain an $\Aone$-homotopy fibre sequence:
  \[ \uHom_\bullet(X,Y) \to \uHom(X,Y) \to Y. \] Since $Y$ is assumed simply-connected, the result may be deduced by
  taking global sections of $\pi_0$ applied to this fibre sequence.
\end{proof}

\section{Stiefel varieties}
\label{sec:stiefel-varieties}

\subsection{Some notation}

We frequently adopt the functor-of-points approach to $k$-schemes. That is, a $k$-scheme $Y$ represents a contravariant functor on the category $\cat{Sch}_k$ of $k$-schemes
\[ h_Y: \cat{Sch}_k^{\op} \to \cat{Set}, \qquad h_Y(X)=\Mor_{\cat{Sch}_k}(X, Y). \]
In fact, we may restrict the source of the functor to the category of affine $k$-schemes, or equivalently to $k$-algebras:
\[ h_Y: k\text{-}\cat{Alg} \to \cat{Set}, \qquad h_Y(R) =\Mor_{\cat{Sch}_k}(\Spec R, Y). \]
The assignment $Y \mapsto h_Y$ yields a full embedding of $\cat{Sch}_k$ in the category of functors $k\text{-}\cat{Alg} \to \cat{Set}$, \cite[Proposition VI-2]{Eisenbud2000}. As a consequence, we will abuse notation and write ``$Y$'' when ``$h_Y$'' is technically correct. Furthermore, we will specify $k$-schemes $Y$ by describing their functors of points $R \mapsto Y(R)$, and will define morphisms of $k$-schemes $Y \to Y'$ by giving the associated natural transformation of functors.

\subsection{Definitions}

If $r \le n$ are two natural numbers, then we define $V_r(\A^n)$ to be the affine $k$-scheme representing the
functor
\begin{equation}
 R \mapsto \left\{ (A, B) \in (\Mat_{r \times n}(R))^2 \mid A B^T  = I_r \right\}.\label{eq:1}
\end{equation}
This is a closed subscheme of $\Mat_{r \times n}^2$, and is therefore affine. We consider $V_r(\A^n)$ as a pointed object of $\sPre(\Sm_k)$ with basepoint given by the $k$-rational point
\[ \left(\begin{bmatrix} I_r & 0 \end{bmatrix}, \begin{bmatrix} I_r & 0 \end{bmatrix} \right). \]

\begin{remark}
  The spaces $V_r(\A^n)$ are a kind of Stiefel variety. The obvious projection furnishes an $\Aone$-equivalence from $V_r(\A^n)$ to the $k$-scheme $V'_r(\A^n)$
  representing the functor
  \begin{equation*}
    R \mapsto \left\{ A \in \Mat_{r \times n}(R) \mid \exists B \in \Mat_{r \times n}(R), \text{ s.t. }A B^T  = I_r \right\},
  \end{equation*}
  which might be considered the true Stiefel variety. There is a morphism $p:V_r(\A^n) \to V'_r(\A^n)$ given by forgetting the choice of $B$. One may cover $V'_r(\A^n)$ by Zariski-open subschemes $U$ so that the morphism $p|_{p^{-1}(U)} : p^{-1}(U) \to U$ is isomorphic to the projections $U \times \A^{(n-r)r} \to U$, i.e., $p$ is a Zariski-locally trivial, smooth morphism with affine-space fibres, and therefore by a standard argument, $p$ is an $\Aone$-equivalence (see e.g., \cite[Lemma 2.4]{Asok2007}).
\end{remark}

By forgetting the bottom $r-r'$ rows, we obtain a pointed morphism $\rho_{r,r'} : V_r(\A^n) \to V_{r'}(\A^n)$. This will be written
$\rho$ when there is no risk of ambiguity.

Two cases of $V_r(\A^n)$ have notation of their own:
\[ V_1(\A^n) = Q_{2n-1}, \quad V_n(\A^n) = \GL(n). \]
We remark that $Q_{2n-1}$ is $\Aone$-homotopy equivalent to $\A^n\sm\{0\}\weq S^{n-1+n\alpha}$.

If we embed $\GL(n-r)$ in $\GL(n)$ in the lower-right position
\[ A \mapsto
  \begin{bmatrix}
    I_r & \\ & A
  \end{bmatrix}
\]
then we arrive at a quotient presentation $V_r(\A^n) = \GL(n)/\GL(n-r)$.

If $\F$ is $\R$ or $\C$, we denote the Stiefel manifold of orthonormal $r$-frames in $\F^n$ by $W_r(\F^n)$.

\subsection{Fibre sequences}

Using results from \cite[Chapter 6]{Morel2012} and \cite[Proposition 8.11]{Morel2012}, we have a diagram of pointed spaces
\begin{equation}
    \GL(n-r) \to \GL(n) \to V_r(\A^n) \overset{f}{\to} B \GL(n-r) \to B\GL(n) \label{eq:5termsequence}
\end{equation}
in which any three consecutive terms form an $\Aone$-homotopy fibre sequence. 

From here, standard homotopy theory, e.g.,
\cite[Proposition 6.3.6]{Hovey1999}, implies that the induced sequence of quotients 
\begin{equation}
  V_s(\A^{n-r+s}) \to V_r(\A^n) \overset{\rho}{\to} V_{r-s}(\A^n)\label{eq:2}
\end{equation}
is an $\Aone$-homotopy fibre sequence whenever $n,r,s$ are integers satisfying $0 \le s <r \le n$.

\begin{proposition} \label{pr:connStiefel}
  The space $V_r(\A^n)$ is $\Aone$-$(n-r-1)$-connected.
\end{proposition}
\begin{proof}
    In the case where $r=1$, we use the $\Aone$-equivalence $V_1(\A^n) \weq S^{n-1+n\alpha}$. The sphere is $n-2$-connected by results of Morel (\cite[Theorem 6.38]{Morel2012}).

    The general result now follows by induction on $r$, using the $s=1$ case of \eqref{eq:2}.
\end{proof}

\begin{corollary} \label{cor:pointedUnpointedStiefel}
  Suppose $n, r$ are natural numbers satisfying $r \le n-2$. Then for any pointed object $X$ of $\sPre(\Sm_k)$, the natural map
  \[ [X, V_r(\A^n)] \to \{X, V_r(\A^n)\} \]
  is a bijection.
\end{corollary}
\begin{proof}
    This follows from Proposition \ref{pr:basepointsMayBeIgnored} using the connectivity calculation of Proposition \ref{pr:connStiefel}.
\end{proof}

\subsection{Interpretation as spaces of stably free modules}
\label{sec:interpr-as-spac}

The $k$-scheme $V_r(\A^n)$ represents a space of matrices as laid out in \eqref{eq:1}. Given a pair of matrices
$(A,B)$ satisfying $AB^T = I_r$, we may form the split short exact sequence of $R$-modules:
\begin{equation}
    \label{eq:4}
    \begin{tikzcd}
      0 \rar & P \rar{\iota} & R^n \rar["A"'] & R^r \arrow[l, "B^T"', bend right]\rar & 0.
    \end{tikzcd}
\end{equation}
Therefore $(A,B)$ determine an $R$-module $P = \ker(A)$, up to isomorphism over $R^n$, along with an isomorphism $P \oplus R^r \to R^n$, given by
$\iota + B^T$. Conversely, given an $R$-module $P$ and an isomorphism $f: P \oplus R^r \to R^n$,
we may produce a morphism $B^T = f|_{R^r}$ and $A = \proj_2 \circ f^{-1}$. This allows us to say that $V_r(\A^n)$
represents the functor that assigns to a ring $R$ the set of equivalence classes of pairs $(P, f)$ where $P$ is an
$R$-module and $f: P \oplus R^r \to R^n$ is an isomorphism: two pairs $(P, f)$ and $(P', f')$ being equivalent when
there exists an isomorphism $h : P \to P'$ for which the diagram
  \begin{equation}
    \label{eq:8}
    \begin{tikzcd}
      P \oplus R^r \arrow[dd, "h \oplus \id_{R^r}"'] \drar{f} & \\ & R^n \\  P' \oplus R^r \urar["f'"']
    \end{tikzcd}
  \end{equation}
is commutative. In this language, the morphism $\rho : V_r(\A^n) \to V_{r'}(\A^n)$ takes a pair $(P,f)$ to $(P\oplus R^{r-r'}, f)$.

\bigskip
The preceding discussion concerns the functor represented by $k$-scheme $V_r(\A^n)$ on the category of commutative $k$-algebras, viz., on affine $k$-schemes. On the category of all $k$-schemes, $V_r(\A^n)$ represents the functor
\begin{equation*}
  X \mapsto \left\{ A \in \Mat_{r \times n}(\Gamma(X,\sh O_X)) \mid \exists B \in \Mat_{r \times n}(\Gamma(X,\sh O_X)), \text{ s.t. }A B^T  = I_r \right\},
\end{equation*}
since $V_r(\A^n)$ is itself affine and therefore $V_r(\A^n)(X) = V_r(\A^n)(\Spec \Gamma(X, \sh O_X))$ by reference to \cite[II, Exercise 2.4]{Hartshorne1977} for instance.

The matrices $A$ and $B$ of global sections of $\sh O_X$ allow us to set up a split short exact sequence
\begin{equation*}
  \begin{tikzcd}
    0 \rar & \sh P \rar{\iota} & \sh O_X^n \rar["A"'] & \sh O_X^r \arrow[l, "B^T"', bend right]\rar & 0,
  \end{tikzcd}
\end{equation*}
as in the affine case,. Therefore, the $k$-scheme $V_r(\A^n)$ represents the functor that assigns to a $k$-scheme $X$ the set of equivalence classes of pairs $(\sh P, f)$ where $\sh P$ is a locally free $\sh O_X$-module and $f: \sh P \oplus \sh O_X^r \to \sh O_X^n$ is an isomorphism. Two pairs $(\sh P, f)$ and $(\sh P', f')$ are equivalent when a commutative diagram analogous to \eqref{eq:8} exists. Note that the sheaves $\sh P$ of $\sh O_X$-modules appearing here are necessarily coherent.

On $V_{r}(\A^{n})$ itself, there exists a tautological $\sh O_X$-module $\sh P_{\taut}$ and a tautological isomorphism $f_{\taut} : \sh P_\taut \oplus \sh O_{V_{r}(\A^{n})}^r \to \sh O^n_{V_{r}(\A^{n})}$. Since $V_{r}(\A^{n})$ is an affine variety, we may alternatively view the above as a tautological module $P_\taut$ and a tautological isomorphism in the category of modules over the coordinate ring of $V_r(\A^n)$.

\section{Homotopy sections}
\label{sec:stably-free-modules}

\begin{definition}
  Consider the morphism $\rho=\rho_{r,1}: V_r(\A^n) \to Q_{2n-1}$. A \emph{homotopy section} of $\rho$ is a morphism $\psi : Q_{2n-1} \to V_r(\A^n)$ in $\cat{H}(k)$ with the property that $\rho \circ \psi = \id_{Q_{2n-1}}$ in $\cat{H}(k)$. Similarly, a \emph{pointed homotopy section} of $\rho$ is a morphism $\phi: Q_{2n-1} \to V_r(\A^n)$ in $\cat{H}(k)_\bullet$ with the property that $\rho \circ \phi = \id_{Q_{2n-1}}$ in $\cat{H}(k)_\bullet$.
\end{definition}

\begin{example}\label{ex:r=2Sec}
  When $n$ is even, there is a well-known section of $\rho_{2,1}$ in the category of $k$-schemes given by
  \[
    ((a_1, \dots, a_n),(b_1, \dots, b_n)) \mapsto 
    \begingroup 
    \setlength\arraycolsep{2pt}
    \left(
    \begin{bmatrix}
      a_1 & a_2 & \cdots & a_{n-1} & a_n \\
      -b_2 & b_1 & \cdots & -b_n & b_{n-1}
    \end{bmatrix},
    \begin{bmatrix}
      b_1 & b_2 & \cdots & b_{n-1} & b_n \\
      -a_2 & a_1 & \cdots & -a_n & a_{n-1}
    \end{bmatrix}
    \right).
    \endgroup
  \]
  This section gives rise to a pointed homotopy section of $\rho_{2,1}$.
\end{example}

\begin{proposition} \label{pr:lindelPopescu}
  Let $R$ be a regular ring of essentially finite type over $k$. Suppose $f_0 , f_1 : \Spec R \to V_r(\A^n)$ are naively homotopic. Write $P_0$ and $P_1$ for the represented stably free modules. Then $P_0 \iso P_1$ as $R$-modules.
 \end{proposition}
\begin{proof}
  Let $H : \Spec R[t] \to V_r(\A^n)$ be the naive homotopy. By pulling the tautological $\sh P_{\taut}$ back to
  $\Spec R[t]$ along $H$, we obtain a stably free $R[t]$-module $P_H$. The stably free modules $P_i$ are then obtained
  as $P_i \iso R \tensor_{R[t]} P_H$ using the two evaluation homomorphisms $e_0: t \mapsto 0$ and $e_1: t \mapsto 1$.

  Since $R$ is regular, the Lindel--Popescu theorem, specifically the main result of \cite{Lindel1981}, implies that
  $P_0 \iso P_1$.
\end{proof}

\begin{proposition} \label{pr:mainTech}
  Suppose $r,n$ are positive integers such that $\rho : V_{r}(\A^{n}) \to Q_{2n-1}$ has a
  homotopy section. Suppose $X$ is a $k$-scheme and $\sh P$ is a sheaf of $\sh O_X$-modules on $X$ with
  the property that $\sh P \oplus \sh O_X \iso \sh O_X^n$. There exists a sheaf $\sh Q$ of $\sh O_X$-modules and an
  isomorphism $ \sh Q \oplus \sh O_X^{r-1} \iso \sh P$.
\end{proposition}
\begin{proof}
  Fix an isomorphism $f: \sh P \oplus \sh O_X \to \sh O_X^{n}$. The pair $(\sh P, f)$ determines a morphism $h: X \to
  Q_{2n-1}$, and $f: \sh P \oplus \sh O_X \to \sh O_X^{n}$ is pulled back from the tautological
  isomorphism over $Q_{2n-1}$. Therefore it suffices to produce $\sh Q_{\taut}$ over $Q_{2n-1}$ so that $\sh
  Q_{\taut} \oplus \sh O_{Q_{2n-1}}^{r-1} \iso \sh P_{\taut}$. That is, we may suppose that $X = Q_{2n-1}$ and $\sh P=\sh P_{\taut}$.

  By hypothesis, there exists a morphism $\psi:Q_{2n-1} \to V_{r}(\A^{n})$ in $\cat{H}(k)$ with the property that
  $\rho\circ \psi = \id_{Q_{2n-1}}$. Using \cite[Theorem 2.4.2]{Asok2018}, we may suppose that $\psi$ is a morphism in
  the category of $k$-schemes, and that there exists a naive $\Aone$-homotopy
  \[ H :  Q_{2n-1} \times \Aone \to Q_{2n-1} \]
  for which $H_0 = \rho \circ \psi$ and $H_1 = \id_{Q_{2n-1} }$.

  The morphism $\psi : Q_{2n-1} \to V_{r}(\A^{n})$ classifies a pair $(\sh Q, g)$ where
  $g : \sh Q \oplus \sh O_{Q_{2n-1}}^r \iso \sh O^{n}_{Q_{2n-1}}$. The composite $\rho \circ \psi$
  classifies the pair $(\sh Q \oplus \sh O_{Q_{2n-1}}^{r-1}, g)$. The existence of the naive $\Aone$ homotopy implies
  that $\sh Q \oplus \sh O_{Q_{2n-1}}^{r-1}$ is isomorphic to $\sh P$, using Proposition \ref{pr:lindelPopescu}.
\end{proof}

\begin{remark} \label{rem:deformHomotopySections}
  The topological analogues of the maps considered above, i.e.,
  $\rho_\C: W_r(\C^n) \to W_1(\C^n)$ and $\rho_\R: W_r(\R^n) \to W_1(\R^n)$, are Serre fibrations. This means that
  homotopy sections of these continuous functions may be deformed to give strict sections, by a lifting argument. The luxury of deforming a homotopy section to a strict section is unavailable to us in the $\Aone$-homotopy theory.
\end{remark}
In spite of the previous remark, the proposition below can be proved.
\begin{proposition}\label{pr:seciffHomotopySec}
  Let $n,r$ be positive integers satisfying $r \le n-2$. The following are equivalent
  \begin{enumerate}
  \item \label{i:1} The morphism $\rho: V_r(\A^n) \to Q_{2n-1}$ has a section in the category of $k$-schemes;
  \item \label{i:2} The morphism $\rho: V_r(\A^n) \to Q_{2n-1}$ has a homotopy section;
  \item \label{i:3} The morphism $\rho: V_r(\A^n) \to Q_{2n-1}$ has a pointed homotopy section.
  \end{enumerate}
\end{proposition}
\begin{proof}
  The implications \ref{i:1} $\Rightarrow$ \ref{i:2} and \ref{i:3} $\Rightarrow$ \ref{i:2} are obvious. 

  (\ref{i:2} $\Rightarrow$ \ref{i:1}). If a homotopy section of $\rho$ exists, then Proposition \ref{pr:mainTech} asserts that the universal projective module $P_{n,n-r}$ has a free summand of rank $r-1$. It follows from \cite[Proposition 2.4]{Raynaud1968} that a section of $\rho$ exists in the category of $k$-schemes. 

  (\ref{i:2} $\Rightarrow$ \ref{i:3}). One may restate Condition \ref{i:3} as saying that the class of the identity map is in the image of 
  \[ [Q_{2n-1}, V_r(\A^n)] \overset{\rho_*}{\longrightarrow} [Q_{2n-1}, Q_{2n-1}].  \]
  Since $1 \le r \le n-2$, both $V_r(\A^n)$ and $Q_{2n-1}$ are $\A^1$-simply connected by Proposition \ref{pr:connStiefel}, so that Corollary \ref{cor:pointedUnpointedStiefel} implies that the vertical arrows in the commuting square below are bijections
  \[
    \begin{tikzcd}
      {[Q_{2n-1},V_r(\A^n)]} \dar["\iso"] \rar["\rho_*"] & {[Q_{2n-1},Q_{2n-1}]} \dar["\iso"] \\
      \{Q_{2n-1},V_r(\A^n)\} \rar["\rho_*"] & \{Q_{2n-1},Q_{2n-1}\}.
    \end{tikzcd}
  \]
  Condition \ref{i:2} asserts that the class of the identity map is in the image of the lower arrow, so that Condition \ref{i:3} follows.
\end{proof}

\section{Realization}
\label{sec:realization}

We refer the reader to \cite{Dugger2004a} for the foundational theory of topological realizations.

We give an overview to fix notation. There are functors
\begin{itemize}
\item Suppose $i: k \hookrightarrow \C$ is an embedding of $k$ in $\C$. There exists \emph{complex realization}, a functor
  $\mathfrak{C} \colon \cat{H}(k)_\bullet \to \cat{H}_\bullet$, which is the composite of $i^*$ and the complex realization of  \cite{Dugger2004a}.
\item Suppose $i: k \hookrightarrow \R$ is an embedding of $k$ in $\R$. There exists \emph{real realization}, a functor
  $\mathfrak{R} : \cat{H}(k)_\bullet \to \cat{H}_\bullet$, which is the composite of $i^*$ and the real realization of  \cite{Dugger2004a}.
\end{itemize}

These realization functors have the following properties:
\begin{itemize}
\item They are compatible with the $\cat{H}_\bullet$-module structure on source and target, i.e., $\mathfrak{C}(K \wedge X)
  \weq K \wedge \mathfrak{C}(X)$ for a pointed simplicial set $K$, and similarly for the real realization.

\item Their values on the quotient schemes $V_r(\A^n) =\GL(n)/\GL(n-r)$ are known. Specifically:
  \[ \mathfrak{C}(V_r(\A^n))\weq W_r(\C^n), \quad  \mathfrak{R}(V_r(\A^n))\weq W_r(\R^n). \]
\item They take the groups $\GL(n)$ to groups.\sebg{to the groups $\U(n)$, $O(n)$, respectively. I suppose this point is subsumed by the previous one.}\benw{Not quite: the point is that the multiplication is preserved. They (strictly) produce $\GL(n, \C)$ and $\GL(n, \R)$.}
\item Their values on the spheres $S^{p+q\alpha}$ are known. Specifically:
   \[ \mathfrak{C}(S^{p+q\alpha})\weq S^{p+q}, \quad  \mathfrak{R}(S^{p+q\alpha})\weq S^p. \]
   Similarly, their values on the motivic Hopf maps of \cite{Dugger2013} are known:
   \begin{gather*} \mathfrak{C}(\eta) =\eta_\top, \quad  \mathfrak{C}(\nu) =\nu_\top; \\
   \mathfrak{R}(\eta) = 2, \quad \mathfrak{R}(\nu) = \eta_\top.\end{gather*}
 \item There are equivalences
   \[  \mathfrak{C}(B\GL(n)) \weq B \GL(n;\C), \quad \mathfrak{R}(B\GL(n)) \weq B \GL(n;\R).\]
\end{itemize}

Assume $p \ge 1$ is an integer. There are homomorphisms of groups, natural in $X$:
\begin{equation}
  \label{eq:5}
  \begin{aligned}
    \bpi_{p+q\alpha}(X)(k) = [S^{p+q\alpha}, X] & \to [S^{p+q}, \mathfrak{C}(X)] = \pi_{p+q}(\mathfrak{C}(X))\\
    \bpi_{p+q\alpha}(X)(k) = [S^{p+q\alpha}, X] & \to [S^{p}, \mathfrak{R}(X)] = \pi_{p}(\mathfrak{R}(X)).
  \end{aligned}
\end{equation}

\subsection{Exactness of realization}

The realization functors we consider do not preserve homotopy fibre sequences in general. Nonetheless, when applied
to the sequences of \eqref{eq:2}, they produce homotopy fibre sequences. As a consequence, we can use realization to
produce commutative diagrams of homotopy groups. We give the formal statement in the case of complex realization only. The other case is similar.

\begin{proposition}
  Complex realization $\mathfrak C$ produces a commutative diagram of long exact sequences of homotopy groups:

  \begin{equation}
    \label{eq:3}
    \adjustbox{scale=0.85}{
     \begin{tikzcd}[column sep=1.4em]
      \cdots \rar  & \bpi_{p+q\alpha}(V_s(\A^{n-r+s}))(k) \dar \rar & \bpi_{p+q\alpha}(V_r(\A^n))(k) \rar \dar & \bpi_{p+q\alpha}(V_{r-s}(\A^n))(k) \arrow[dr, phantom, "\dagger"] \rar[pos=.4,"\bd"] \dar & \bpi_{p-1+q\alpha}(V_s(\A^{n-r+s}))(k) \rar \dar  & \cdots \\
      \cdots \rar  & \pi_{p+q}(W_s(\C^{n-r+s})) \rar & 
      \pi_{p+q}(W_r(\C^n))\rar & \pi_{p+q}(W_{r-s}(\C^s))\rar{\bd} & \pi_{p+q-1}(W_s(\C^{n-r+s})) \rar  & \cdots.
      \end{tikzcd}}
  \end{equation}
  
\end{proposition}
\begin{proof}
  The lower sequence is exact because there is a homotopy fibre sequence
  \[ W_s(\C^n) \to W_r(\C^n) \to W_{r-s}(\C^n). \]
  In \eqref{eq:3}, commutativity of the squares other than the one marked with a dagger $\dagger$ follows from the
  naturality of the homomorphism in \eqref{eq:5}. The square marked with the dagger is not \textit{a priori} induced by
  a map of schemes. Nonetheless, it can be factored
  \[
  \adjustbox{scale=0.9}{
\begin{tikzcd}[column sep=1.2em]
   \bpi_{p+q\alpha}(V_{r-s}(\A^n))(k) \dar \rar &  \bpi_{p+q\alpha}(B\GL(n-r+s))(k) \dar \rar[pos=.35,"\iso"] \arrow[dr, phantom, "\ast"]  &  \bpi_{p-1+q\alpha}(\GL(n-r+s))(k) \dar \rar & \bpi_{p-1+q\alpha}(V_s(\A^{n-r+s}))(k) \dar \\
    \pi_{p+q}(W_{r-s}(\C^n)) \rar &  \pi_{p+q}(B\GL(n-r+s; \C))  \rar{\iso} &  \pi_{p+q-1}(\GL(n-r+s;\C))  \rar & \pi_{p+q-1}(W_s(\C^{n-r+s})).
\end{tikzcd}}
\]
The identification $\bpi_{p+q\alpha}(B\GL(n-r+s))(k)= \bpi_{p-1+q\alpha}(\GL(n-r+s))(k)$ is made as follows: there is a
canonical map
\begin{equation} \label{eq:canonical}
   S^1 \wedge G \to BG
\end{equation}
so that any morphism $S^{p+q\alpha} \to G$ yields a morphism $S^{p+1+q\alpha} \to S^1 \wedge G \to BG$ by
composition. When $G= \GL(n-r+s)$, the adjoint to $S^1 \wedge G \to BG$ is an $\Aone$-weak equivalence by \cite[Theorem
6.46]{Morel2012}, so that indeed we obtain the asserted identification.

  Complex realization of \eqref{eq:canonical} yields a similar map $S^1 \wedge G(\C) \to BG(\C)$, so that the square
  marked with an asterisk $\ast$ also commutes.
\end{proof}

\section{Constant homotopy sheaves}
\label{sec:sh-f_k-constant}

Fix a subfield $i: k \into \C$. Throughout this section, $n$ denotes a natural number.

\begin{definition} \label{def:constInjSurj} Suppose $X$ is a pointed motivic space over $k$, and $\bpi_{p + q\alpha}(X)$
  is a strictly $\A^1$-invariant homotopy sheaf. We will say that $\bpi_{p+q\alpha}(X)$ has the \emph{constant}
  (resp.~\emph{surjective}, \emph{injective}) \emph{realization property} if the map
  \[ \bpi_{p+q\alpha}(X)(k) \to \pi_{p+q}(\mathfrak{C}(X)) \]
  is an isomorphism (resp.~surjective, injective).
\end{definition}

\begin{remark}
    If instead we fix a subfield $i : k \into \R$, we may define the \emph{constant} (resp.~\emph{surjective}, \emph{injective}) \emph{real realization property} by using $\mathfrak R$ instead of $\mathfrak C$.
\end{remark}




\begin{example} \label{ex:1stemAC}
 Let $n \ge 4$. The homotopy sheaves $\bpi_{n}(Q_{2n-1})$ may be
   partly calculated from the main theorem of \cite{Rondigs2019} and \cite[Theorem 6.3.6]{Asok2024}. There is an exact sequence
     \begin{equation}
     \begin{tikzcd}
       0 \rar & \K^\M_{n+2}/(24) \rar &  \bpi_{n}(Q_{2n-1}) \rar &  \bpi^s_{1-n\alpha}(\kq) 
     \end{tikzcd}\label{eq:6}
   \end{equation}
   which becomes short exact after $n-4$-fold contraction. Furthermore, $\bpi^s_{1-3\alpha}(\kq)$ is identified with a
   sheaf $\GW^3_4$ as defined in \cite{Schlichting2015} or \cite{Asok2017b} (see \cite[Diagram 3.9]{Asok2020}).

   Therefore, once we contract \eqref{eq:6} $n+1$-times, we obtain a short exact sequence
   \begin{equation}
     \begin{tikzcd}
       0 \rar & \K^\M_1/(24) \rar & \bpi_{n+(n+1)\alpha}(Q_{2n-1}) \rar{q} & \GW^3_0 \rar & 0 
     \end{tikzcd}
     \label{eq:7}
   \end{equation}
   using the identities $(\GW^i_j)_{-1} = \GW^{i-1}_{j-1}$ and $\GW^i_j = \GW^{i+4}_j$. By \cite[Theorem
   10.1]{Walter2003}, the sheaf $\GW^3_0$ is constant on fields, with value $\eta \eta_\top\Z/(2)$---see for instance
   \cite[p.~58]{Rondigs2019}.

   The complex realization of $\eta \eta_\top$ is $\eta_\top^2$, which generates $\pi_{2n+1}(S^{2n-1})$, so that
   $\bpi_{n+(n+1)\alpha}(Q_{2n-1})$ has the surjective realization property.

 \end{example}
 
 \begin{remark}
   If we assume further that $k$ is quadratically and cubically closed, so that $\K^\M_1(k)/(24) \iso 0$, then
   $\bpi_{n+(n+1)\alpha}(Q_{2n-1})$ has the constant realization property.
 \end{remark}
 
 \begin{example} \label{ex:1stemB}
   We may contract \eqref{eq:7} one more
   time. Then we obtain
   \[
     \begin{tikzcd}
       0 \rar & \Z/(24) \rar & \bpi_{n+(n+2)\alpha}(Q_{2n-1}) \rar & 0,
     \end{tikzcd}
   \]
   so that $ \bpi_{n+(n+2)\alpha}(Q_{2n-1})$ is constant with value $\Z/(24)$. It has the constant
   realization property, being generated by the class of $\nu$, which realizes to give $\nu_\top$, which also
   generates $\pi_{2n+2}(S^{2n-1})$.
 \end{example}

 \begin{example}\label{ex:pi2Real}
    Let $n \ge 5$. The homotopy sheaves $\bpi_{n+1}(Q_{2n-1})$ may be
   partly calculated from the main theorem of \cite{Rondigs2024} and \cite[Theorem 6.3.6]{Asok2024}. There is an exact sequence
     \begin{equation*}
     \begin{tikzcd}
       0 \rar & H^{1+n, 2+n}(-)/(24) \oplus \K^\M_{4+n}/(2) \rar & \bpi_{n+1}(Q_{2n-1}) \rar & \bpi^s_{2-n\alpha}(\kq),
     \end{tikzcd}
   \end{equation*}
   and after $n+2$ contractions, this yields an isomorphism
    \begin{equation}
     \label{eq:8b}
     \begin{tikzcd}
       \K^\M_2/(2) \rar{\iso} & \bpi_{n+1+(n+2)\alpha}(Q_{2n-1}).
     \end{tikzcd}
   \end{equation}
   If $k \into \R$ is a field with a real embedding, then the class of $\{-1,-1\}$ in $\K^\M_2(k)/(2) \iso \Hoh_\et^2(k; \Z/(2))$ corresponds to
   $\rho^2\nu^2 \in \bpi_{n+1+(n+2)\alpha}(Q_{2n-1})(k)$. Similarly, after $n+3$ or $n+4$ contractions, the classes of
   $\{-1\} \in \K^\M_1(k)/(2)$ and $1 \in \Z/(2)$ correspond to $\rho \nu^2$ and $\nu^2$.

   Since the real realization of $\rho$ is the identity, and the real realization of $\nu$ is $\eta_\top$, we
   deduce that $\bpi_{n+1+(n+d)\alpha}(Q_{2n-1})$ has the constant real realization property when $d \in \{2,3,4\}$ and  $\K^\M_{4-d}(k)/(2)
   \iso \Z/(2)$.
 \end{example}
 
\section{The existence of homotopy sections}
\label{sec:sections}

Consider the homotopy long exact sequence of
\[ V_{r-1}(\A^{n-1}) \overset{i}{\to} V_{r}(\A^{n}) \overset{\rho}{\to} Q_{2n-1}. \] 
A portion of this sequence appears below:
\begin{equation}
  \label{eq:9}
  \begin{tikzcd}
    \cdots \rar & \bpi_{n-1}(V_{r}(\A^{n})) \rar{\rho_*} & \bpi_{n-1}(Q_{2n-1})\arrow[d,equal] \rar{\bd} & \bpi_{n-2}(V_{r-1}(\A^{n-1}))
    \rar & \cdots. \\
    & & \K^{\MW}_n
  \end{tikzcd}
\end{equation}

\begin{notation} \label{not:defBeta}
Following \cite{James1958}, we denote the image of the identity map under 
\[ 
  \bd_{-n}(k): \bpi_{n-1 + n\alpha}(Q_{2n-1})(k) = [Q_{2n-1},Q_{2n-1}] \to \bpi_{n-2+n\alpha}(V_{r-1}(\A^{n-1}))(k) 
\] 
by $\beta_r^n$ and call this element \emph{the obstruction}.
\end{notation}
The following proposition justifies this terminology.
\begin{proposition} \label{pr:equivOfSections}
  Let $n$ and $r$ be integers with $2 \le r \le n-2$. The following are equivalent:
  \begin{enumerate}
  \item \label{i:i} The morphism $\rho_{1,r}: V_r(\A^n) \to Q_{2n-1}$ admits a section in the category of $k$-schemes;
  \item \label{i:ii} The morphism $\rho_{1,r}: V_r(\A^n) \to Q_{2n-1}$ admits a pointed homotopy section;
  \item The map \[ {\rho_{1,r}}_* : \bpi_{n-1}(V_{r}(\A^{n})) \to \bpi_{n-1}(Q_{2n-1}) = \K^{\MW}_n \] is surjective;
  \item In the homotopy long exact sequence \eqref{eq:9}, the boundary map $\bd: \bpi_{n-1}(Q_{2n-1}) \to \bpi_{n-2}(V_{r-1}(\A^{n-1}))$ vanishes;
  \item \label{i:v} The obstruction $\beta_r^n$ vanishes.
  \end{enumerate}
\end{proposition}
\begin{proof}
  The equivalence of \ref{i:i} and \ref{i:ii} is given by Proposition \ref{pr:seciffHomotopySec}, and the other forward implications are trivial. It remains to show \ref{i:v} $\Rightarrow$ \ref{i:ii}. 
  
  Suppose the obstruction vanishes. Using \cite[Lemma 5.1.3]{Asok2017a}, we may identify $\bd$ with a specific element of $\bpi_{n-2+n\alpha}(V_{r-1}(\A^{n-1}))(k)$, which is tautologically $\bd_{-n}(k)(\id) = \beta_r^n = 0$. That is, $\bd=0$ as a morphism of sheaves. 
  
  To construct a pointed homotopy section, we argue as follows.  There is an $\Aone$-homotopy fibre sequence
  \[ \Omega V_r(\A^n) \to \Omega Q_{2n-1} \overset{f}{\to} V_{r-1}(\A^{n-1}) \]
  and the map $\bd$ may be obtained by applying $\bpi_{n-2+n\alpha}$ to $f$. Letting $j$ denote the adjoint of the identity map in $\mathbf H(k)_\bullet$, we obtain a homotopy commutative diagram
  \[
    \begin{tikzcd}[column sep=4em]
      & S^{n-2+n\alpha} \dar{j} \arrow[dr, "\bd"] \arrow[dl, dashed] \\ \Omega V_{r}(\A^{n}) \arrow[r,"\Omega(\rho_{1,r})"'] & \Omega Q_{2n-1} \rar["f"'] & V_{r-1}(\A^{n-1}).
    \end{tikzcd}
  \]
  Since $\bd$ is null, the dashed arrow may be constructed, and by adjunction, a pointed homotopy section of $\rho_{r,1}$ exists.
\end{proof}

\begin{proposition} \label{pr:reductionToRealize} Let $n$ and $r$ be integers with $2 \le r \le n-2$. Let $i : k \into \C$ be a fixed embedding.  Suppose that $\bpi_{n-2+n\alpha}(V_{r-1}(\A^{n-1}))$
  has the injective realization property and that $\rho(\C): W_{r}(\C^{n}) \to S^{2n-1}$ has a section. Then
  $\rho: V_{r}(\A^{n}) \to Q_{2n-1}$ has a section.
\end{proposition}
\begin{proof}
  Entirely analogously to Proposition \ref{pr:equivOfSections}, one deduces that
  $\rho(\C): W_r(\C^n) \to S^{2n-1}$ has a section if and only if the boundary homomorphism
  $ \bd(\C): \pi_{2n-1}(S^{2n-1}) \to \pi_{2n-2}(W_{r-1}(\C^{n-1}))$ is $0$, i.e., if it takes the class of the identity
  in $\pi_{2n-1}(S^{2n-1}) = \Z$ to $0$.
  
  There is, therefore, a commutative diagram
  \[
    \begin{tikzcd}[column sep=5em]
     \bpi_{n-1+n\alpha}(Q_{2n-1})(k) \rar{\bd_{-n}(k)} \dar{\mathfrak C} &  \bpi_{n-2+n\alpha}(V_{r-1}(\A^{n-1}))(k) \arrow[d,
      "\mathfrak{C}", hook]\\
    \pi_{2n-1}(S^{2n-1}) \rar{\bd=0} &  \pi_{2n-2}(W_{r-1}(\C^{n-1})).
    \end{tikzcd}
  \]
  From which it follows that $\bd_{-n}(k)=0$, and therefore that $\rho$ has a section, by Proposition
  \ref{pr:equivOfSections}.
\end{proof}

\subsection{The case \texorpdfstring{$r=3$}{r=3}}

\begin{proposition} \label{pr:r=3New}
  Let $n \ge 3$ be an integer. Let $k=\Q$. Then $\rho: V_{3}(\A^{n}) \to Q_{2n-1}$ has a section if and only if $n \equiv 0 \pmod{24}$.
\end{proposition}
\begin{proof}
  Unless $n$ is a positive multiple of $b_3=24$, then the method of \cite[Th\'eor\`eme 6.5]{Raynaud1968} prevents
  $\rho_{3,1}$ from having a section. Therefore, we assume $n$ is a positive
  multiple of $24$. In this case, \cite{James1958} asserts that a section of $W_3(\C^n) \to S^{2n-1}$ exists. Therefore we need only
  prove that $\pi_{n-2+n\alpha}(V_2(\A^{n-1}))$ has the injective realization property. In fact, it has the constant
  realization property.

  We consider the long exact sequence of homotopy sheaves
  \[
    \begin{tikzcd}[column sep = small]
    & \arrow[d,phantom,""{coordinate, name=a}] & \bpi_{n-1+n\alpha}(Q_{2n-3}) \ar[dll, "\bd"', rounded corners,
    to path={ -- ([xshift=1ex]\tikztostart.east)
            |- (a) [pos=1]\tikztonodes
            -| ([xshift=-2ex]\tikztotarget.west)
		-- (\tikztotarget)}] \\
    \bpi_{n-2+n\alpha}(Q_{2n-5}) \rar & \bpi_{n-2+n\alpha}(V_2(\A^{n-1})) \rar  \arrow[d,phantom,""{coordinate, name=b}]  &
    \bpi_{n-2+n\alpha}(Q_{2n-3}) = \K^{\MW}_{-2} 
    \ar[dll, "\times \eta"', "\iso", rounded corners,
    to path={ -- ([xshift=1ex]\tikztostart.east)
            |- (b) [pos=1]\tikztonodes
            -| ([xshift=-2ex]\tikztotarget.west)
    	-- (\tikztotarget)}] \\
    \bpi_{n-3+n\alpha}(Q_{2n-5}) = \K^{\MW}_{-3} & \phantom{A} &
    \end{tikzcd}
  \]
  The first sheaf appearing has the surjective realization property (Example \ref{ex:1stemAC}) and the next sheaf has
  the constant (\textit{a fortiori}, injective) realization property (Example \ref{ex:1stemB}). The last two sheaves are both isomorphic to the
  Witt sheaf by results of \cite{Morel2012} and the map between them is the isomorphism $\times \eta$, by \cite[Lemma 3.5]{Asok2014}.

  It follows from an easy diagram chase that $\bpi_{n-2+n\alpha}(V_2(\A^{n-1})) $ has the constant realization
  property. We conclude by Proposition \ref{pr:reductionToRealize}.
\end{proof}

\subsection{The case \texorpdfstring{$r=4$}{r=4}}

First we need a technical lemma concerning the obstruction.

\begin{lemma}\label{lem:sMapsTob}
  Let $n$ and $r$ be integers with $2 \le r \le n$, and suppose that $\psi: Q_{2n-1} \to V_{r-1}(\A^n)$ is a pointed homotopy section. Then the obstruction $\beta_r^n$ is the image of $\psi$ under the composition
  \[
    \bpi_{n-1+n\alpha}(V_{r-1}(\A^n))(k) \overset{\bd_{-n}(k)}{\longrightarrow} \bpi_{n-2+n\alpha}(Q_{2(n-r+1)-1}))(k) \overset{(i_*)_{-n}(k)}{\longrightarrow} \bpi_{n-2+n\alpha}(V_{r-1}(\A^{n-1}))(k).
  \]
\end{lemma}
\begin{proof}
  There is a map of $\A^1$-homotopy fibre sequences
  \[
    \begin{tikzcd}
      Q_{2(n-r+1)-1} = V_1(\A^{n-r+1}) \dar["i"] \rar["i"] & V_r(\A^n) \dar[equals] \rar["\rho"] & V_{r-1}(\A^n) \dar["\rho"] \\
      V_{r-1}(\A^{n-1}) \rar["i"] & V_r(\A^n) \rar["\rho"] & Q_{2n-1}
    \end{tikzcd}
  \]
  In particular, the diagram
  \[
    \begin{tikzcd}[column sep=5em]
      \bpi_{n-1+n\alpha}(V_{r-1}(\A^n))(k) \dar["(\rho_*)_{-n}(k)"] \rar["\bd_{-n}(k)"] & \bpi_{n-2+n\alpha}(Q_{2(n-r+1)-1})(k) \dar["(i_*)_{-n}(k)"] \\
      \bpi_{n-1+n\alpha}(Q_{2n-1})(k) \rar["\bd_{-n}(k)"] & \bpi_{n-2+n\alpha}(V_{r-1}(\A^{n-1}))(k)
    \end{tikzcd}
  \]
  commutes. Then
  \[
    ((i_*)_{-n}(k) \circ \bd_{-n}(k))(\psi) = (\bd_{-n}(k) \circ (\rho_*)_{-n}(k))(\psi) = \bd_{-n}(k)(\id_{V_1(\A^n)}) = \beta_r^n,
  \]
  as desired.
\end{proof}

\begin{remark}
    In Proposition \ref{pr:r=4} below, the case of quadratically closed $k$ is redundant, because of the argument of Remark \ref{rem:dontNeedQuadClosed}. We include this case here because the proof is short, and we expect the same proof to apply when $k$ is a quadratically closed field of characteristic greater than $3$.
\end{remark}

\begin{proposition} \label{pr:r=4}
  Let $n \geq 4$ be an integer and $k$ be a subfield of $\C$ satisfying one of the following:
  \begin{enumerate}
    \item $k$ is quadratically closed. 
    \item $k$ is a subfield of $\R$ and admits a unique quadratic extension (up to isomorphism).
  \end{enumerate}
  Then the map $\rho_{4,1}: V_4(\A^n) \to Q_{2n-1}$ has a section if and only if $n \equiv 0 \pmod{24}$.
\end{proposition}
\begin{proof}
  As in the $r=3$ case, the method of \cite[Th\'eor\`eme 6.5]{Raynaud1968} allows us to assume that $n$ is a positive multiple of $b_4=24$. Let $\psi: V_1(\A^n) \to V_3(\A^n)$ be a pointed homotopy section (which exists by Proposition \ref{pr:r=3New}), and consider the sequence
  \begin{equation}\label{eq:r=4Ims}
    \bpi_{n-1+n\alpha}(V_3(\A^{n}))(k) \xrightarrow{\bd_{-n}(k)} \bpi_{n-2+n\alpha}(Q_{2(n-3)-1})(k) \xrightarrow{(i_*)_{-n}(k)} \bpi_{n-2+n\alpha}(V_3(\A^{n-1}))(k).
  \end{equation}
  We claim that $(\bd_{-n}(k))(\psi) = 0$. With the help of Proposition \ref{pr:equivOfSections} and Lemma \ref{lem:sMapsTob}, the result follows from this claim. Note that the middle group in \eqref{eq:r=4Ims} is in the stable range (Example \ref{ex:pi2Real}) and coincides with the corresponding stable homotopy group of the motivic sphere spectrum  $\bpi_{2+3\alpha} (\one)(k) = \K_1^M(k)/(2)$ (compare \eqref{eq:8b}).
  
  If $k$ is quadratically closed, we have $\K_1^M(k)/(2) = 0$ by assumption so there is nothing left to be done.
  
  For the second case, real realization induces the commuting diagram
  \begin{equation}\label{eq:r=4Real}
  \begin{tikzcd}[column sep=4.5em]
    \bpi_{n-1+n\alpha}(V_3(\A^{n}))(k) \rar["\bd_{-n}(k)"] \dar["\mathfrak{R}"] & \bpi_{n-2+n\alpha}(Q_{2(n-3)-1}))(k) \rar["(i_*)_{-n}(k)"] \dar["\mathfrak{R}"] & \bpi_{n-2+n\alpha}(V_3(\A^{n-1}))(k) \dar["\mathfrak{R}"] \\
    \pi_{n-1}(W_3(\R^{n})) \rar["\bd"] & \pi_{n-2}(S^{n-4}) \rar["{i_\R}_*"] & \pi_{n-2}(W_3(\R^{n-1}))
  \end{tikzcd}
  \end{equation}
  Under the left vertical map, the class of the section $\psi: Q_{2n-1} \to V_3(\A^n)$ is taken to the class of a
  section $S^{n-1} \to W_3(\R^n)$. It follows from \cite[2.1]{James1958} that the element
  \[ {i_\R}_* \circ \bd \circ \mathfrak{R}(\psi) = ({i_\R}_* \circ \mathfrak{R} \circ \bd_{-n}(k))(\psi) \] is the obstruction to the
  existence of a section of $\rho_\R: W_4(\R^n) \to S^{n-1}$. This obstruction vanishes by \cite[Proposition
  1.1]{James1957}. Moreover, the source and target of the middle realization map in \eqref{eq:r=4Real} are in the stable
  range, so by the assumption and Example \ref{ex:pi2Real}, the middle realization map is an isomorphism. Consequently,
  to prove the claim, it suffices to show that ${i_\R}_*: \pi_{n-2}(S^{n-4}) \to \pi_{n-2}(W_3(\R^{n-1}))$ is injective.
  
  A portion of the long exact sequence in homotopy groups associated with the homotopy fibre sequence
  \[
    S^{n-4} \xrightarrow{i_\R} W_3(\R^{n-1}) \xrightarrow{\rho(\R)} W_2(\R^{n-1})
  \]
  is
  \[
    \pi_{n-2}(S^{n-4}) \xrightarrow{{i_\R}_*} \pi_{n-2}(W_3(\R^{n-1})) \xrightarrow{{\rho_\R}_*} \pi_{n-2}(W_2(\R^{n-1})) \overset{\bd}{\to} \pi_{n-3}(S^{n-4}).
  \]
  The isomorphisms $\pi_{n-2}(W_3(\R^{n-1})) \cong \Z/2 \oplus \Z/2$ and $\pi_{n-2}(W_2(\R^{n-1})) \cong \Z/(2)$ can be read from the tables of \cite{Paechter1956}, from which we conclude that ${i_\R}_*$ is injective.
\end{proof}

\section{The main result}

Proposition \ref{pr:mainTech} and Propositions \ref{pr:r=3New} and \ref{pr:r=4} have the following immediate consequence, which is the main result of this paper.

\begin{theorem}\label{th:schThryMain}
  Suppose $X$ is a scheme over a subfield $k$ of $\C$ and $\sh P$ is a sheaf of $\sh O_X$-modules with the property $\sh P \oplus \sh O_X \iso \sh O_X^{24n}$ for some positive integer $n$. Then there exists a sheaf of $\sh O_X$-modules $\sh Q$ such that 
  \[ \sh P \iso \sh Q \oplus \sh O_X^2. \] 
  Suppose further that $k$ is quadratically closed or is a subfield of $\R$ that admits a unique quadratic extension (up to isomorphism). Then there is a sheaf of $\sh O_X$-modules $\sh Q'$ and an isomorphism 
  \[ \sh P \iso \sh Q' \oplus \sh O_X^3. \]
\end{theorem}

\printbibliography
\end{document}


%% file: tbjwHeader.tex
\usepackage{etoolbox}
\usepackage{ifdraft}

\usepackage{xcolor}

\usepackage[colorlinks=true, citecolor=green!50!black, linkcolor=red!50!black, pagebackref=false, final]{hyperref}

\ifdraft{\usepackage[color,notref, notcite]{showkeys}}{}
\usepackage[letterpaper]{geometry}
\usepackage{xargs}
\usepackage{amsmath, amsthm, mathrsfs, amssymb}
\usepackage{tikz-cd} 
\usepackage[final]{graphicx}
\usepackage{enumitem}
\usepackage{bm}

\usepackage[en-AU]{datetime2}
\usepackage[T1]{fontenc}

\newcommand{\mathbx}[1]{\mathbb{#1}}

\usepackage[colorinlistoftodos,prependcaption,textsize=tiny]{todonotes} 

\theoremstyle{plain}
\ifcscounter{chapter}{%
        \newtheorem{thm}{Theorem}[chapter]%
        }%
        {%
        \newtheorem{thm}{Theorem}[section]%
        }
\newtheorem{theorem}[thm]{Theorem}

\newtheorem{lemma}[thm]{Lemma}

\newtheorem{corollary}[thm]{Corollary}

\newtheorem{proposition}[thm]{Proposition}

\newtheorem*{claim*}{Claim} 
\newtheorem*{thm*}{Theorem}
\newtheorem*{theorem*}{Theorem}

\theoremstyle{definition}

\newtheorem{definition}[thm]{Definition}

\newtheorem{notation}[thm]{Notation}

\theoremstyle{remark}

\newtheorem{remark}[thm]{Remark}

\newtheorem{example}[thm]{Example}

\newtheorem{question}[thm]{Question}

\newcommand{\C}{\mathbx{C}}

\newcommand{\Q}{\mathbx{Q}}
\newcommand{\R}{\mathbx{R}}
\newcommand{\Z}{\mathbx{Z}}

\newcommand{\A}{\mathbx{A}}
\newcommand{\F}{\mathbx{F}}

\newcommand{\FF}{\F}

\newcommand{\aone}{\A^1}
\newcommand{\Aone}{\aone}

\newcommand{\M}{\mathsf{M}}

\newcommand{\Gl}{\operatorname{GL}}
\newcommand{\Mat}{\operatorname{Mat}}

\newcommand{\U}{\operatorname{U}}

\newcommand{\cat}[1]{\text{\bf #1}}

\newcommand{\tensor}{\otimes}
\newcommand{\into}{\hookrightarrow}

\newcommand{\spec}{\operatorname{Spec}}

\newcommand{\Spec}{\spec}

\newcommand{\Sm}{\cat{Sm}}

\newcommand{\sh}[1]{\mathcal{#1}}

\newcommand{\weq}{\simeq}

\newcommand{\sPre}{\cat{sPre}}

\newcommand{\op}{\text{op}}

\newcommand{\GL}{\Gl}

\newcommand{\Hoh}{\mathrm{H}}

\renewcommand{\top}{\mathrm{top}}

\newcommand{\id}{\mathrm{id}}
\newcommand{\isom}{\cong}
\newcommand{\iso}{\isom}

\newcommand{\sm}{\setminus}

\newcommand{\bd}{\partial}

\newcommand{\proj}{\operatorname{proj}}

\newcommand{\MW}{\mathsf{MW}}

\newcommand{\et}{\textup{\'et}}

\newcommand{\bpi}{\bm{\pi}}

\newcommand{\K}{{{\mathbf K}}}

\newcommand{\Mor}{\operatorname{Mor}}
\DeclareMathOperator{\uHom}{\underline{Hom\kern-.05em}\kern.1em}
